\documentclass[onefignum,onetabnum]{siamart190516}

\definecolor{officegreen}{rgb}{0.0, 0.5, 0.0}
\newcommand{\YY}[2][red]{{\textcolor{#1}{#2}}}

\usepackage{lipsum}
\usepackage{amsfonts}
\usepackage{graphicx}
\usepackage{epstopdf}
\usepackage{algorithmic}
\usepackage{hyperref}
\usepackage{bm}
\usepackage{comment}
\usepackage{multirow}

\newcommand{\mcH}{\mathcal{H}}
\newcommand{\mcU}{\mathcal{U}}
\newcommand{\mcO}{\mathcal{O}}

\newcommand{\mcX}{\mathcal{X}}
\newcommand{\mcN}{\mathcal{N}}
\newcommand{\mcL}{\mathcal{L}}

\newcommand{\mbR}{\mathbb{R}}
\newcommand{\mbRd}{{\mathbb{R}^d}}

\newcommand{\mcLn}{\mathcal{L}^{N\!L}}
\newcommand{\mcLp}{\mathcal{L}^{L\!P\!S}}

\newcommand{\omg}{{\Omega}}
\newcommand{\omgi}{{\Omega_I}}

\newcommand{\oomg}{{\overline\Omega}}

\newcommand{\ds}{\displaystyle}

\def \bb{\mathbf{b}}

\def \nub{{\boldsymbol \nu}}

\def \sbb{\mathbf{s}}
\def \ub{\mathbf{u}}
\def \wb{\mathbf{w}}
\def \vb{\mathbf{v}}
\def \xb{\bm{x}}
\def \yb{\bm{y}}

\def \zerob{\mathbf{0}}

\def \sl{{s_l}}
\def \ul{{u_l}}
\def \un{{u_n}}

\def \wun{{\widetilde u_n}}

\ifpdf
  \DeclareGraphicsExtensions{.eps,.pdf,.png,.jpg}
\else
  \DeclareGraphicsExtensions{.eps}
\fi

\newsiamremark{remark}{Remark}
\newsiamremark{hypothesis}{Hypothesis}
\crefname{hypothesis}{Hypothesis}{Hypotheses}
\newsiamthm{claim}{Claim}

\headers{Prescription of boundary conditions for nonlocal problems}{M. D'Elia and Y. Yu}

\title{On the prescription of boundary conditions for nonlocal diffusion and peridynamics models %\thanks{\funding{M. D'Elia was supported by xxx. Y. Yu was supported by yyy.}}
}

\author{Marta D'Elia\thanks{Sandia National Laboratories, Livermore, CA 
  (\email{mdelia@sandia.gov}).}
\and Yue Yu\thanks{Lehigh University, Bethlehem, PA 
  (\email{yuy214@lehigh.edu}).}}

\usepackage{amsopn}

\ifpdf
\hypersetup{
  pdftitle={On the prescription of boundary conditions for nonlocal Poisson's and peridynamics models},
  pdfauthor={M. D'Elia and Y. Yu}
}
\fi

\begin{document}

\maketitle

\begin{abstract}
We introduce a technique to automatically convert local boundary conditions into nonlocal volume constraints for nonlocal Poisson's and peridynamic models. The proposed strategy is based on the approximation of nonlocal Dirichlet or Neumann data with a local solution obtained by using available boundary, local data. The corresponding nonlocal solution converges quadratically to the local solution as the nonlocal horizon vanishes, making the proposed technique asymptotically compatible. The proposed conversion method does not have any geometry or dimensionality constraints and its computational cost is negligible, compared to the numerical solution of the nonlocal equation. The consistency of the method and its quadratic convergence with respect to the horizon is illustrated by several two-dimensional numerical experiments conducted by meshfree discretization for both the Poisson's problem and the linear peridynamic solid model. 
\end{abstract}

\begin{keywords}
Nonlocal models, peridynamics, nonlocal boundary conditions, convergence to local limits, asymptotic behavior of solutions, meshfree discretization.
\end{keywords}

\begin{AMS}
34B10, 35B40, 45A05, 45K05, 65M75, 74A70, 76R50.
\end{AMS}

%%%%%%%%%%%%%%%%%%%%%%%%%%%%%%%%%%%%%%%%%%%%%%%%%%%%%%%%%%%%%%%%%%%%%
%%%%%%%%%%%%%%%%%%%%%%%%%%%%%%%%%%%%%%%%%%%%%%%%%%%%%%%%%%%%%%%%%%%%%
\section{Introduction and motivation}\label{sec:introduction}
Nonlocal, integral models are valid alternatives to classical partial differential equations (PDEs) to describe systems where small scale effects or interactions affect the global behavior. In particular, nonlocal models are characterized by integral operators that embed length scales in their definitions, allowing to capture long-range space interactions. Furthermore, the integral nature of such operators reduces the regularity requirements on the solutions that are allowed to feature discontinuous or singular behavior. Applications of interest span a large spectrum of scientific and engineering fields, including
fracture mechanics \cite{Ha2011,Silling2000}, anomalous subsurface transport \cite{Benson2000,Gulian2021,Schumer2003,Schumer2001}, phase transitions \cite{Burkovska2021,Delgoshaie2015,Fife2003}, image processing \cite{Buades2010,DElia2021Imaging,Gilboa2007}, magnetohydrodynamics \cite{Schekochihin2008}, stochastic processes \cite{Burch2014,DElia2017,Meerschaert2012,MeKl00}, and turbulence \cite{DiLeoni2021,Pang2020npinns,Pang2019fPINNs}.

Despite their improved accuracy, the usability of nonlocal equations is hindered by several modeling and computational challenges that are the subject of very active research. Modeling challenges include the lack of a unified and complete nonlocal theory \cite{Defterli2015,DElia2020Helmholtz,DElia2020Unified}, the nontrivial treatment of nonlocal interfaces \cite{Alali2015,Capodaglio2020,Seleson2013,fan2021asymptotically,you2020asymptotically,yu2018partitioned} and the non-intuitive prescription of nonlocal boundary conditions \cite{DEliaNeumann2019,You_2019,trask2019asymptotically,yu2021asymptotically,Foss2021}. Computational challenges are due to the integral nature of nonlocal operators that yields discretization matrices that feature a much larger bandwidth compared to the sparse matrices associated with PDEs. For both variational methods \cite{Aulisa2021,Capodaglio2021DD,DElia-ACTA-2020,DEliaFEM2020} and meshfree methods \cite{Pasetto2019,Silling2005meshfree,Wang2010,XuFETI2021,trask2019asymptotically,You_2019,you2020asymptotically,fan2021asymptotically,yu2021asymptotically} a lot of progress has been made during the last decade, resulting in improved numerical techniques that facilitate wider adoption, even at the engineering level. 

In its simplest form, the action of a nonlocal (spatial) operator on a scalar function $u:\mbRd\to\mbR$ is defined as 
\begin{equation*}
    \mcL u(\xb) = \int_{\mcH_\delta(\xb)} I(\xb,\yb,u)\,d\yb,
\end{equation*}
where $\mcH_\delta(\xb)$ defines a nonlocal neighborhood of size $\delta$ surrounding a point $\xb\in\mbRd$, $d$ being the spatial dimension and $\delta$ the so-called horizon or interaction radius. The latter defines the extent of the nonlocal interactions and embeds the nonlocal operator with a characteristic length scale. The integrand function $I$ is application dependent and plays the role of a constitutive law. Its definition is not straightforward and represents one of the most investigated problems in nonlocal research \cite{Burkovska2021Learning,DElia2016ParamControl,Xu2021,You2021,You2020Regression}.

In this work we focus on the prescription of nonlocal boundary conditions, or volume constraints, when solving nonlocal equations in bounded domains. The challenge stems from the presence of nonlocal interactions, for which a point $\xb$ in a domain interacts with points outside of the domain that are contained in the point's neighborhood $\mcH_\delta(\xb)$. This fact generates an interaction region of nonzero measure where volume constraints need to be prescribed to guarantee the uniqueness of a nonlocal solution \cite{Du2013}. However, often times, input data to a problem are not available (due to measurement cost or physical impediments) in volumetric regions, whereas they are only available on the surfaces surrounding the domain. In other words, the only available data are {\it local}. Thus, the question arises of {\it how to convert local boundary information into a nonlocal volume constraint.} 

In the nonlocal literature, this issue has been addressed in several works, most of which propose conversion approaches that are either too restrictive (in terms of geometry or dimensionality constraints), too computationally expensive (requiring the solution of an optimization problem), or are not prone to wide usability (requiring a modification of available codes). Among these works we mention \cite{DEliaCoupling,You_2019,yu2021asymptotically,Foss2021}. 

The method we propose is inspired by the recent work \cite{DEliaNeumann2019} where the authors propose to first approximate the nonlocal solution with its local counterpart and then {\it correct} it by solving the nonlocal problem using the local solution to generate volume constraints. In \cite{DEliaNeumann2019} Neumann local boundary conditions are converted into Dirichlet or Neumann volume constraints in the context of nonlocal Poisson's problems and numerical tests are performed in one dimension. Based on this work we propose to convert Dirichlet local boundary conditions into Dirichlet or Neumann volume constraints in the context of both nonlocal Poisson's and peridynamics equations. Furthermore, we show applicability of our strategy in a two-dimensional setting using nontrivial geometries. 

The main idea of the proposed method can be summarized in three simple steps. 
\begin{enumerate}
\item Using available local data, we solve the local counterpart of the nonlocal problem. This step assumes that the local limit (the limit as $\delta\to 0$) of the nonlocal operator is known\footnote{Local limits of nonlocal operators can be obtained by using Taylor's expansion; both the nonlocal Poisson's problem and the peridynamic model considered in this work have well-known local limits, namely, the (local) Poisson's equation and the Navier equation of linear elasticity, respectively.}, that the local data is smooth enough to guarantee well-posedness, and that a solver for the corresponding local equation is available.
\item We use the local solution to define either the nonlocal Dirichlet data in the nonlocal interaction domain or to obtain the nonlocal Neumann data by computing the corresponding nonlocal flux. This step numerically corresponds to a matrix-vector multiplication and does not require the implementation of a new nonlocal (flux) operator; in fact, as we will explain later, the nonlocal Neumann operator is the nonlocal operator itself evaluated at points in the nonlocal interaction domain. \label{item:conversion}
\item Use either the Dirichlet or Neumann data obtained in {Step} \ref{item:conversion}. to solve the nonlocal problem, for which volume constraints are now available. 
\end{enumerate}

\smallskip \noindent
We summarize the main properties of the proposed approach below.
\begin{itemize}
    \item This strategy delivers a nonlocal solution that is physically consistent with PDEs in the limit of vanishing nonlocality. Numerically, {when employing proper numerical discretization methods, e.g., the optimization-based meshfree quadrature rule \cite{trask2019asymptotically,yu2021asymptotically},} this property guarantees {\it asymptotic compatibility} \cite{Tian2014}, i.e. the nonlocal numerical solution converges to its local limit as $\delta$ and the discretization size $h$ approach 0.
    \item This technique has no geometry or dimensionality constraints. It can be utilized with any domain shape and in all dimensions $d=1,2,3$.
    \item The conversion of local data into nonlocal volume constraints is inexpensive. In fact, it corresponds to a matrix-vector product where the matrix is either a selection matrix (in the Dirichlet case) or a nonlocal flux matrix (in the Neumann case).
    \item This strategy does not require the implementation of new software. In fact, available local and nonlocal solvers can be used as black boxes. 
\end{itemize}
Consequently, this strategy has the potential of dramatically increasing the usability of nonlocal models at the engineering and industry level thanks to its flexibility, intuitiveness, and ease of implementation.

\smallskip \noindent{\bf Paper outline} This paper is organized as follows. In the following section we describe the nonlocal Poisson's and linear peridynamic solid (LPS) models. For each of them, we introduce the strong and weak formulations and discuss conditions for their well-posedness. In Section \ref{sec:strategy} we illustrate the proposed strategies for the conversion of a local, Dirichlet boundary condition into a nonlocal Dirichlet (DtD strategy) or Neumann (DtN strategy) volume constraint. In Section \ref{sec:local-limit} we prove that both approaches deliver nonlocal solutions that are asymptotically compatible with the corresponding local solution of both the Poisson's and LPS problems. Specifically we prove that the nonlocal solution converges to the local one with quadratic rate. In Section \ref{sec:tests} we illustrate the properties of our methods with several two-dimensional numerical tests. In particular, we show that when the solutions are such that local and nonlocal operators are equivalent our procedure satisfies the consistency property (the nonlocal solution coincides with the local one). Furthermore, for both models and both approaches we confirm the quadratic convergence rate of the $L^2$-norm difference between local and nonlocal solutions. Finally, in Section 
\ref{sec:conclusion} we summarize our achievements.

%%%%%%%%%%%%%%%%%%%%%%%%%%%%%%%%%%%%%%%%%%%
%%%%%%%%%%%%%%%%%%%%%%%%%%%%%%%%%%%%%%%%%%%
\section{Preliminaries}\label{sec:preliminaries}
In this section we introduce the mathematical models used in this paper and recall relevant results. In what follows, scalar fields are indicated by italic symbols and vector fields by bold symbols. 
Let $\omg$ be a bounded open domain in $\mbRd$, $d=1,2,3$, with Lipschitz-continuous boundary $\partial\omg$. 

%%%%%%%%%%%%%%%%%%%%%%%%%%%%%%%%%%%%%%%%%%%
\subsection{The nonlocal Poisson's problem}\label{sec:poisson}
For the function $u(\xb)\colon \mbRd \to \mbR$ we define the nonlocal Laplacian $\mcLn\colon \mbRd \to \mbR$ of $u(\xb)$ as
\begin{equation}\label{eq:L}
\mcLn u(\xb) :=  2\int_\mbRd \big(u(\yb)-u(\xb)\big) \,\gamma (\xb,\yb )\,d\yb \qquad  \xb\in\mbRd,
\end{equation}
where  $\gamma(\xb,\yb)$ is a nonnegative symmetric kernel\footnote{For more general, sign-changing and nonsymmetric kernels we refer the reader to \cite{mengesha2013analysis} and \cite{DElia2017}, respectively.} such that, for $\xb\in\omg$
\begin{equation}\label{eq:kernel}
\left\{
\begin{array}{ll}
\gamma(\xb,\yb)  >  0 \quad &\forall\, \yb\in B_\delta(\xb)\\[2mm]
\gamma(\xb,\yb)  = 0 \quad &\forall\, \yb\in \mbRd \setminus B_\delta(\xb),
\end{array}\right.
\end{equation}
where $B_\delta(\xb) = \{\yb\in\mbRd: \; \|\xb-\yb\|<\delta,\; \xb\in \omg\}$ and $\delta$ is the interaction radius or horizon. For the Laplacian operator $\mcLn$, we define the interaction domain of $\omg$ associated with kernels like in \eqref{eq:kernel} as follows
\begin{equation}\label{eq:omgi}
\omgi = \{ \yb\in \mbRd\setminus\omg: \; \|\yb-\xb\|<\delta, \;\text{ {for some} }\xb\in\omg\},
\end{equation}
and set $\oomg =\omg\cup\omgi$. The domain $\omgi$ contains all points outside of $\omg$ that interact with points inside of $\omg$; as such, $\omgi$ is the volume where nonlocal boundary conditions, or volume constraints, must be prescribed to guarantee the well-posedness of the nonlocal equation associated with $\mcLn$ \cite{Du2013}. We refer to Figure \ref{fig:domains} (left) for an illustration of a two-dimensional domain, the support of $\gamma$ and the induced interaction domain. Here, the interaction domain is divided into the nonoverlapping partition $\omgi=\omg_{nloc}\cup \omg_{loc}$. In what follows we assume that nonlocal data is available on $\omg_{nloc}$ whereas only local information is available on the physical boundary of $\omg_{loc}$, i.e. on $\Gamma_{loc}=\partial\omg_{loc}\cap\partial\omg$.
%%%
\begin{figure}
\centering
\includegraphics[width=0.45\textwidth]{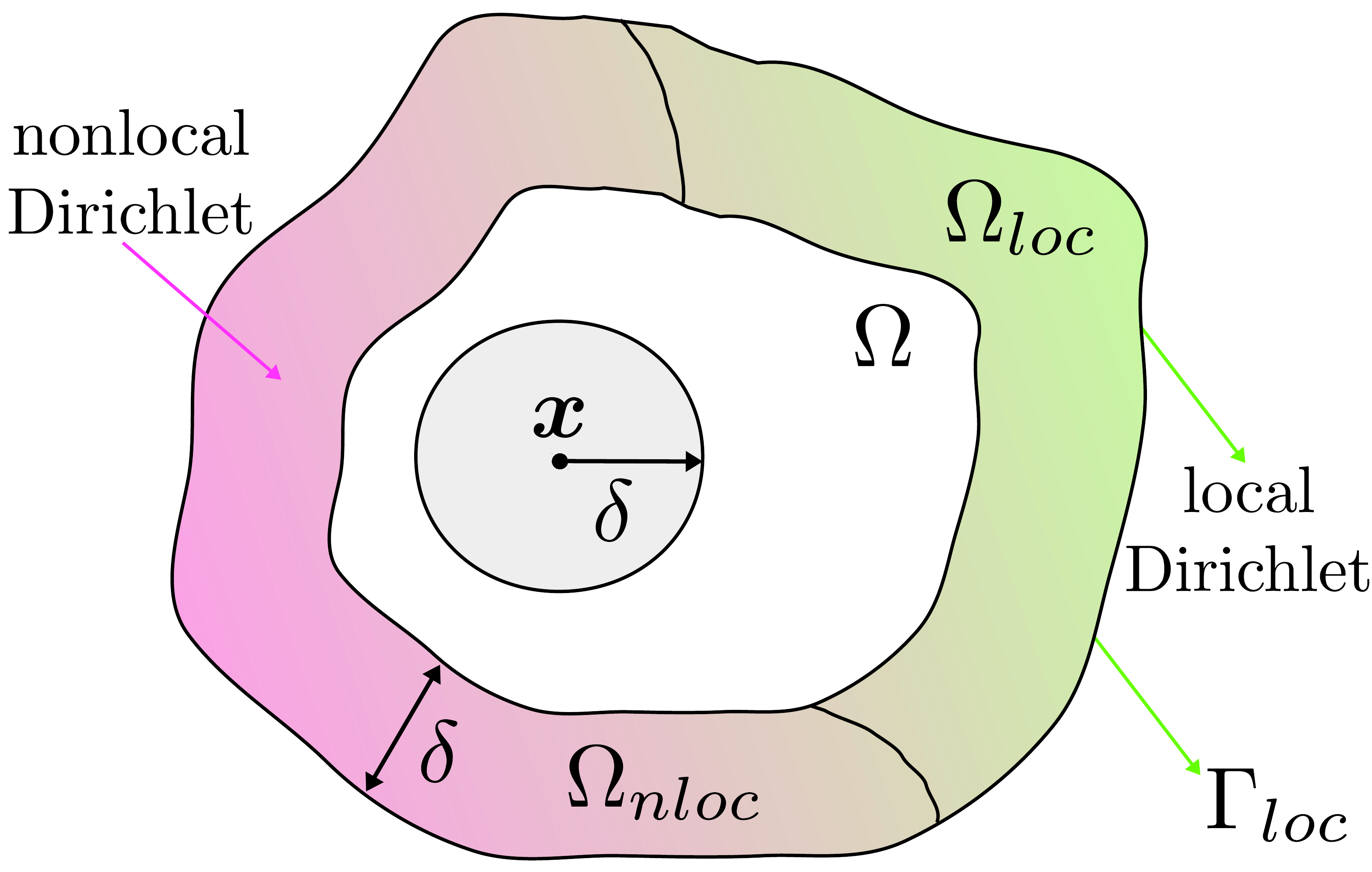}
\includegraphics[width=0.45\textwidth]{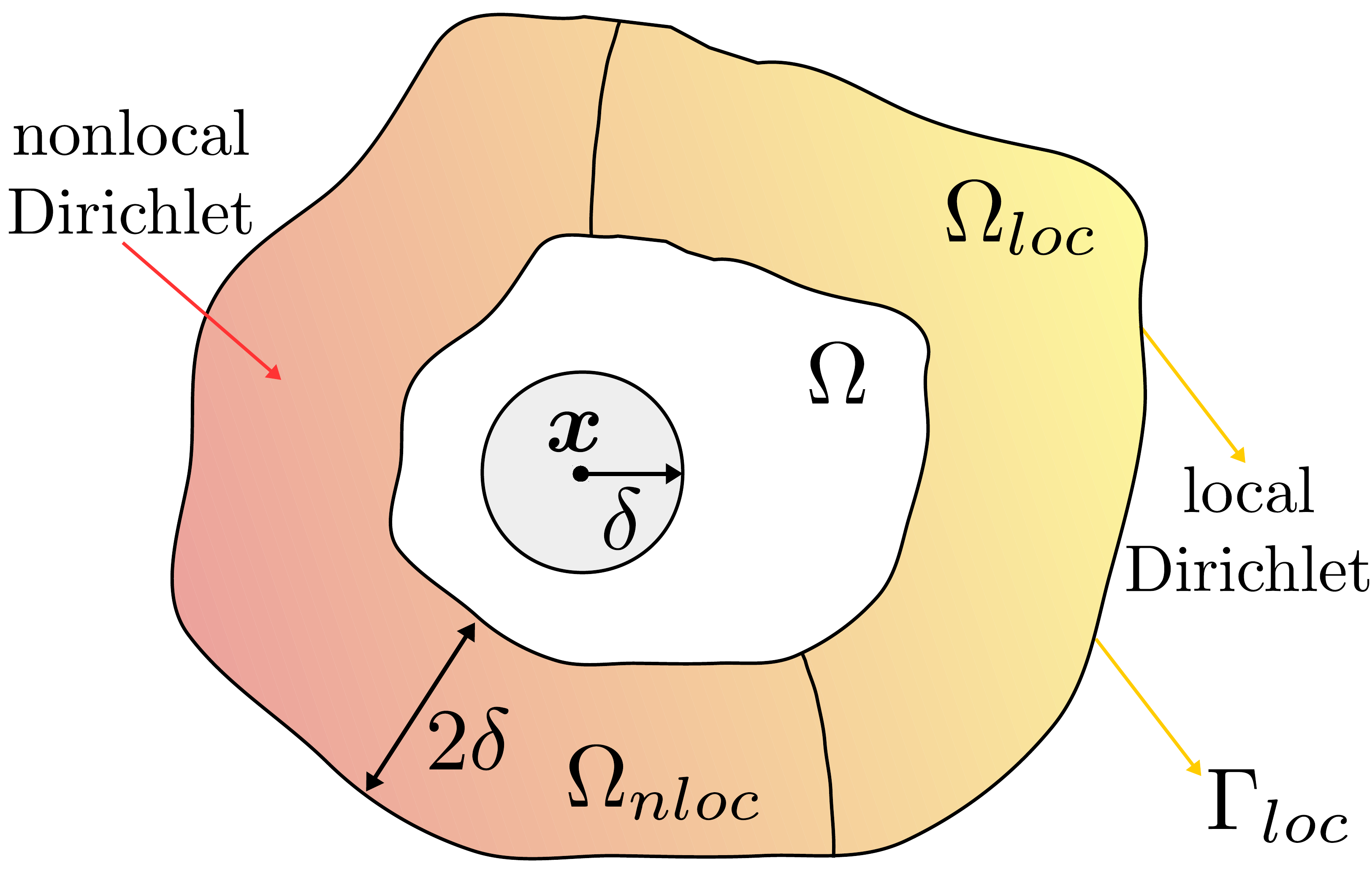}
\caption{The domain $\omg$, the support of $\gamma$ at a point $\xb\in\omg$, $B_\delta(\xb)$, and the induced interaction domain $\omgi$ for the nonlocal Poisson's problem (left) and the LPS model (right).}
\label{fig:domains}
\end{figure}
%%%

An important property of the Laplacian operator in \eqref{eq:L} is its $\delta$-convergence, i.e. as $\delta\to 0$ to the classical, local Laplacian $\Delta$. In fact, when the kernel $\gamma$ is properly scaled, we have the following pointwise relationship:
\begin{equation}\label{eq:nl-op-difference}
\mcL u(\xb) = \Delta u(\xb) + \mcO(\delta^2).  
\end{equation}

With the purpose of prescribing Neumann volume constraints, we introduce the nonlocal {\it flux} operator:
\begin{equation*}\label{eq:flux}
\mcN^{N\!D} u(\xb)=-\int_\oomg (u(\yb)-u(\xb))\gamma(\xb,\yb)\,d\yb \qquad \xb\in\omgi.
\end{equation*}
To provide an interpretation of the interaction operator, we note that the integral $\int_\omgi \mcN^{N\!D}(\nub)\,d\xb$ generalizes the concept of a local flux $\int_{\partial\omg}{\bf q}\cdot{\bf n}\,dA$ through the boundary of a domain, with $\mcN(\nub)$ being the nonlocal counterpart of the local flux density ${\bf q}\cdot{\bf n}$. We refer to \cite{Du2013} for additional details regarding the nonlocal vector calculus and results such as integration by parts and nonlocal Green's identities.

We introduce the nonlocal energy semi-norm, nonlocal energy space, and nonlocal volume-constrained energy space
\begin{equation}
\begin{array}{ll}
& |||v|||^2    := \displaystyle\int_{\oomg}\int_{{\oomg}}(u(\yb)-u(\xb))^2\gamma(\xb,\yb)\,d\yb \, d\xb \\ [5mm]
& V(\oomg)   := \left\{ v  \in L^2(\oomg) \,\,:\,\, |||v|||_{\oomg} < \infty \right\}\\[3mm]
& V_\Lambda(\oomg) := \left\{v\in V({\oomg}) \,\,:\,\, v=0\;{\rm on}\;\Lambda\subset\Omega_I\right\}.
\end{array}
\end{equation}
We also define the volume-trace space $\widetilde V_{\Lambda}(\oomg):=\{v|_\Lambda: \,v\in V(\oomg)\}$, for $\Lambda\subset\omgi$, and the dual spaces $V'(\oomg)$ and $V'_\Lambda(\oomg)$ with respect to $L^2$-duality pairings.

We consider kernels such that the corresponding energy norm satisfies a Poincar\'e-like inequality, i.e. $\|v\|_{0,{\oomg}}\leq C_{pn}|||v|||$ for all $v\in V_\Lambda(\oomg)$, where  $C_{pn}$ is the nonlocal Poincar\'e constant. For such kernels, the paper \cite{DuMengesha} shows that $C_{pn}$ is independent of $\delta$ if $\delta\in (0, \delta_0]$ for a given $\delta_0$. In this paper we consider a specific class of kernels, namely, integrable kernels such that there exist positive constants $\gamma_1$ and $\gamma_2$ for which $\gamma_1\leq \int_{\oomg} \gamma(\xb,\yb)\,d\yb$ and $\int_\oomg \gamma^2(\xb,\yb)\,d\yb\leq \gamma_2^2$ for all $\xb\in\omg$. In this setting $V(\oomg)$ and $V_\Lambda(\oomg)$ are equivalent to $L^2({\oomg})$ and $L^2_c(\oomg)$ and the operator $\mcL$ is such that $\mcL:L^2(\oomg)\to L^2(\oomg)$ \cite{Du2012}.

%%%%%%%%%%%%%%%%%%
\medskip\noindent{\bf Strong form} We introduce the strong form of a nonlocal Poisson's problem with Dirichlet or mixed volume constraints. We refer, again, to the configuration in Figure \ref{fig:domains} (left) and recall that $\omgi=\Omega_{nloc}\cup\Omega_{loc}$ such that $\Omega_{nloc}\cap\Omega_{loc}=\emptyset$. For $s\in V'(\oomg)$, $v_n\in \widetilde V_{\omg_{nloc}}(\oomg)$, and $w_n\in \widetilde V_{\omg_{loc}}(\oomg)$ we define the {\it Dirichlet Poisson's problem} as: find $\un\in V(\oomg)$ such that 
\begin{equation}\label{eq:nonlocal-dirichlet}
\left\{\begin{array}{rlll}
-\ds\mcLn\un & = & s   & \xb\in\omg  \\[3mm]
\un     &    = & w_n & \xb\in\Omega_{loc} \\[3mm]
\un     &    = & v_n & \xb\in\Omega_{nloc},
\end{array}\right.
\end{equation} 
where \eqref{eq:nonlocal-dirichlet}$_2$ and \eqref{eq:nonlocal-dirichlet}$_3$ are two distinct Dirichlet volume constraints.
Similarly, given $s\in V'(\oomg)$, $v_n\in \widetilde V_{\omg_{nloc}}(\oomg)$, and $g_n\in V'(\omg_{loc})$, we define the {\it mixed Poisson's problem} as follows: find $\un\in V(\oomg)$ such that 
\begin{equation}\label{eq:nonlocal-mixed}
\left\{\begin{array}{rlll}
-\ds\mcLn\un & = & s   & \xb\in\omg  \\[3mm]
-\mcN \un     &    = & g_n & \xb\in\Omega_{loc} \\[3mm]
\un     &    = & v_n & \xb\in\Omega_{nloc},
\end{array}\right.
\end{equation} 
where \eqref{eq:nonlocal-mixed}$_2$ is the nonlocal counterpart of a flux condition, i.e. a Neumann boundary condition. As such, we refer to it as Neumann volume constraint.

%%%%%%%%%%%%%%%%%%
\medskip\noindent{\bf Weak form}
With the purpose of analyzing the $\delta$-convergence of our strategies, we also introduce the weak form of problems \eqref{eq:nonlocal-dirichlet} and \eqref{eq:nonlocal-mixed}. By multiplying both equations by a test function and using nonlocal integration by parts \cite{Du2012}, we obtain the following weak formulations. 

For $s\in V'(\oomg)$, $v_n\in \widetilde V_{\omg_{nloc}}(\oomg)$, and $w_n\in \widetilde V_{\omg_{loc}}(\oomg)$ we define the {\it Dirichlet Poisson's problem} as: find $\un\in V_c(\oomg)$ such that $\un=w_n$ in $\omg_{loc}$, $\un=v_n$ in $\omg_{nloc}$ and, for all $z\in V(\oomg)$,
\begin{equation}\label{eq:weak-dirichlet}
\begin{aligned}
\ds\int_\oomg&\int_\oomg (\un(\xb)-\un(\yb))(z(\xb)-z(\yb))\gamma(\xb,\yb)\,d\yb\,d\xb =
\int_\omg sz\,d\xb,\\[3mm]
\end{aligned}
\end{equation}
or, equivalently, $a(u,z) = F(z)$, where the bilinear form is given by $a(u,z)=\langle u,z\rangle_{V_\omgi}$. It can be shown \cite{Du2012} that for every $\gamma(\cdot,\cdot)$ satisfying the Poincar\'e inequality $a(\cdot,\cdot)$ is coercive and continuous in $V_\omgi(\oomg)\times V_\omgi(\oomg)$ and that $F(\cdot)$ is continuous in $V_\omgi(\oomg)$. Thus, by the Lax-Milgram theorem problem \eqref{eq:weak-dirichlet} is well-posed.

Similarly, given $s\in V'(\oomg)$, $v_n\in \widetilde V_{\omg_{nloc}}(\oomg)$, and $g_n\in V'(\omg_{loc})$, one can define the {\it mixed Poisson's problem} as follows: find $\un\in V(\oomg)$ such that $\un=v_n$ in $\omg_{nloc}$ and for all $z\in V_{\omg_{nloc}}(\oomg)$,
\begin{equation}\label{eq:weak-mixed}
\begin{aligned}
\ds\int_\oomg&\int_\oomg (\un(\xb)-\un(\yb))(z(\xb)-z(\yb))\gamma(\xb,\yb)\,d\yb\,d\xb =
\int_{\omg_{loc}} g_n z\,d\xb +
\int_\omg sz\,d\xb,
\end{aligned}
\end{equation}
or, equivalently, $a(u,z)=F_{g_n}(z)$. Also in this case, it can be shown that $a(\cdot,\cdot)$ is coercive and continuous in $V_{\omg_{nloc}}(\oomg)$, provided the kernel induces a Poincar\'e inequality. Furthermore, the functional $F_{g_n}$ is continuous on $V_{\omg_{nloc}}(\oomg)$. Thus, by the Lax-Milgram theorem problem \eqref{eq:weak-mixed} is also well-posed.

%%%%%%%%%%%%%%%%%%%%%%%%%%%%%%%%%%%%%%%%%%%
\subsection{The linear peridynamic solid model}\label{sec:lps-model}
For the displacement function $\ub(\xb)\colon \mbRd \to \mbRd$, we define the linear peridynamic solid (LPS)  \cite{emmrich2007well} operator\footnote{Note that this model holds in the assumption of small displacements \cite{emmrich2007well}.} $\mcLp\colon \mbRd \to \mbRd$ as 
\begin{equation}\label{eq:LPS}
\begin{aligned}
\mcLp \ub(\xb) := & \dfrac{C_1}{m(\delta)} \int_\oomg \left(\lambda- \mu\right) \gamma(\left|\mathbf{y}-\mathbf{x}\right|) \left(\mathbf{y}-\mathbf{x} \right)\left(\theta(\mathbf{x}) + \theta(\mathbf{y}) \right) d\mathbf{y}\\
+ &
\frac{C_2}{m(\delta)}\int_\oomg \mu \gamma(\left|\mathbf{y}-\mathbf{x}\right|)\frac{\left(\mathbf{y}-\mathbf{x}\right)\otimes\left(\mathbf{y}-\mathbf{x}\right)}{\left|\mathbf{y}-\mathbf{x}\right|^2}  \left(\mathbf{u}(\mathbf{y}) - \mathbf{u}(\mathbf{x}) \right) d\mathbf{y},
\end{aligned}
\end{equation}
where the dilatation $\theta:\mbRd\to\mbR$ is defined as 
\begin{equation*}
    \theta(\mathbf{x}):=\dfrac{2}{m(\delta)}\int_\oomg \gamma(\left|\mathbf{y}-\mathbf{x}\right|) (\mathbf{y}-\mathbf{x})\cdot \left(\mathbf{u}(\mathbf{y}) - \mathbf{u}(\mathbf{x}) \right)d\mathbf{y}.
\end{equation*}
Here, for $d=2$, $C_1=2$ and $C_2=16$. The kernel function $\gamma$ is nonnegative and radial and satisfies the same assumptions as in \eqref{eq:kernel}. Furthermore, we consider kernels $\gamma$ such that $m$, defined as
$$
m(\delta):=\int_{B_\delta (\mathbf{x})}\gamma(\left|\mathbf{y}-\mathbf{x}\right|)\left|\mathbf{y}-\mathbf{x}\right|^2d\mathbf{y},
$$
is bounded. This guarantees well-posedness of the volume constrained problem associated with $\mcLp$ \cite{mengesha14Navier}.
The constants $\mu$, $\lambda$ are the shear and Lam\'e modulus, that, under the plane strain assumption \cite{bower2009applied}, are related to the Young's modulus $E$ and the Poisson ratio $\nu$ of a material, i.e. $\lambda=\frac{E\nu}{(1+\nu)(1-2\nu)}$, $\mu=\frac{E}{2(1+\nu)}$. It can be shown \cite{mengesha14Navier} that the LPS operator $\mcLp$ converges to the Navier operator below
\begin{equation}\label{eq:local-lps}
\mcL_l \mathbf{u}:=-\nabla\cdot(\lambda tr(\mathbf{E})\mathbf{I}+2\mu \mathbf{E})=-(\lambda-\mu)\nabla [\text{tr}(\mathbf{E})]-\mu \nabla\cdot(2\mathbf{E}+\text{tr}(\mathbf{E})\mathbf{I}),
\end{equation}
where $\mathbf{E}:=\dfrac{1}{2}(\nabla \mathbf{u}+(\nabla \mathbf{u})^T)$ and $\text{tr}(\mathbf{E})=\nabla\cdot\mathbf{u}$. In particular, we have the following pointwise relationship
\begin{equation}\label{eq:limit-lps}
    \mcLp \ub(\xb) = \mcL_l \ub(\xb) +\mcO(\delta^2).
\end{equation}

For the LPS operator $\mcLp$, we define the interaction domain of $\omg$ as
\begin{equation}\label{eq:omgi-2delta}
\omgi = \{ \yb\in \mbRd\setminus\omg: \; \|\yb-\xb\|<2\delta, \;\xb\in\omg\}.
\end{equation}
and set $\oomg =\omg\cup\omgi$. Note that in this case, $\omgi$ is a layer of thickness $2\delta$ surrounding $\omg$; this is due to the presence of a double integral in the definition of the operator. As before, $\omgi$ is the volume where nonlocal boundary conditions must be prescribed to guarantee the well-posedness of the nonlocal equation associated with $\mcLp$. We refer to Figure \ref{fig:domains} (right) for an illustration of a two-dimensional domain, the support of $\gamma$ and the induced interaction domain. The same division as in Section \ref{sec:poisson} into a nonoverlapping partition is performed. 

For the prescription of nonlocal flux conditions, we consider the following nonlocal flux operator for the LPS model. Let $\xb\in\Lambda\subset\omgi$, we have
\begin{equation}\label{eq:N-LPS}
\begin{aligned}
\mcN^{L\!P\!S} \ub(\xb) := & \dfrac{C_1}{m(\delta)} \int_\oomg \left(\lambda- \mu\right) \gamma(\left|\mathbf{y}-\mathbf{x}\right|) \left(\mathbf{y}-\mathbf{x} \right)\left(\theta(\mathbf{x}) + \theta(\mathbf{y}) \right) d\mathbf{y}\\
+&
\frac{C_2}{m(\delta)}\int_\oomg \mu \gamma(\left|\mathbf{y}-\mathbf{x}\right|)\frac{\left(\mathbf{y}-\mathbf{x}\right)\otimes\left(\mathbf{y}-\mathbf{x}\right)}{\left|\mathbf{y}-\mathbf{x}\right|^2}  \left(\mathbf{u}(\mathbf{y}) - \mathbf{u}(\mathbf{x}) \right) d\mathbf{y},
\end{aligned}
\end{equation}
where $\theta$ and $m$ are defined as above. For more details on nonlocal flux conditions for nonlocal mechanics problems we refer the interested reader to \cite{lehoucq2008force}.

As for the Laplacian operator, we introduce the energy norm and the corresponding spaces \cite{mengesha14Navier}.
\begin{equation}
\begin{aligned}
|||\ub|||^2_{L\!P\!S}=&  \frac{1}{m(\delta)}\int_\oomg \int_{\oomg\cap B_\delta (\mathbf{x})}  \dfrac{\gamma(\left|\mathbf{y}-\mathbf{x}\right|)}{\left|\mathbf{y}-\mathbf{x}\right|^2}\left[\left(\mathbf{u}(\mathbf{y}) - \mathbf{u}(\mathbf{x}) \right)\cdot\left(\mathbf{y}-\mathbf{x}\right)\right]^2 d\mathbf{y}\,d\mathbf{x}, \\[2mm]
V^{L\!P\!S}(\oomg):=&\left\{\ub\in [L^2(\oomg)]^d:|||\ub|||_{L\!P\!S}<\infty \right\} \\[2mm]
V^{L\!P\!S}_\Lambda(\oomg):=&\left\{\ub\in V^{L\!P\!S}(\oomg):\ub={\bf 0} \; {\rm on} \; \Lambda\subset\omgi\right\}
\end{aligned}
\end{equation}
Note that $|||\ub|||_{L\!P\!S}=0$ if and only if $\ub$ represents an infinitesimally rigid displacement, i.e.:
$$\ub(\xb)\in\{\mathbb{Q}\xb+\bb,\mathbb{Q}\in\mathbb{R}^{d\times d},\mathbb{Q}^T=-\mathbb{Q},\bb\in\mathbb{R}^d\}.$$
We also define the volume-trace space $\widetilde V^{L\!P\!S}_{\Lambda}(\oomg):=\{v|_\Lambda: \,v\in V^{L\!P\!S}(\oomg)\}$, for $\Lambda\subset\omgi$, and the dual spaces $(V^{L\!P\!S})'(\oomg)$ and $(V^{L\!P\!S})'_\Lambda(\oomg)$ with respect to $L^2$-duality pairings. Note that when $\gamma$ is an integrable function, similarly to the nonlocal Laplacian operator, the LPS operator acts as a map from $[L^2(\oomg)]^d$ to $[L^2(\oomg)]^d$.

%%%%%%%%%%%%%%%%%%
\medskip\noindent{\bf Strong form} We introduce the strong form of the LPS problem with Dirichlet or mixed volume constraints. We refer, again, to the configuration in Figure \ref{fig:domains} (right). For $\sbb\in (V^{L\!P\!S})'(\oomg)$, $\vb_n\in \widetilde V^{L\!P\!S}_{\omg_{nloc}}(\oomg)$, and $\wb_n\in \widetilde V^{L\!P\!S}_{\omg_{loc}}(\oomg)$ we define the {\it Dirichlet LPS problem} as: find $\ub_n\in V^{L\!P\!S}(\oomg)$ such that 
\begin{equation}\label{eq:lps-dirichlet}
\left\{\begin{array}{rlll}
-\ds\mcLp\ub_n & = & \sbb   & \xb\in\omg  \\[3mm]
\ub_n     &    = & \wb_n & \xb\in\Omega_{loc} \\[3mm]
\ub_n     &    = & \vb_n & \xb\in\Omega_{nloc},
\end{array}\right.
\end{equation} 
where \eqref{eq:lps-dirichlet}$_2$ and \eqref{eq:lps-dirichlet}$_3$ are distinct Dirichlet volume constraints.
Similarly, given $\sbb\in (V^{L\!P\!S})'(\oomg)$, $\vb_n\in \widetilde V^{L\!P\!S}_{\omg_{nloc}}(\oomg)$, and ${\bf g}_n\in (V^{L\!P\!S})'(\omg_{loc})$, we define the {\it mixed LPS problem} as follows: find $\ub_n\in V^{L\!P\!S}(\oomg)$ such that 
\begin{equation}\label{eq:lps-mixed}
\left\{\begin{array}{rlll}
-\ds\mcLp\ub_n & = & \sbb   & \xb\in\omg  \\[3mm]
-\mcN^{L\!P\!S} \ub_n     &    = & {\bf g}_n & \xb\in\Omega_{loc} \\[3mm]
\ub_n     &    = & \vb_n & \xb\in\Omega_{nloc}.
\end{array}\right.
\end{equation}

\medskip\noindent{\bf Weak form} With the purpose of analyzing the $\delta$-convergence of our strategies, we also introduce the weak form of problems \eqref{eq:lps-dirichlet} and \eqref{eq:lps-mixed}. For clarity, and to avoid heavy notation, we present the formulations in the scalar setting.
We first introduce the following integration by parts result \cite{DElia2020Unified,Du2012}: for every $u\in V^{L\!P\!S}(\oomg)$ and $z\in V^{L\!P\!S}_{\omg_{nloc}}(\oomg)$, we have
\begin{equation}\label{eq:by-parts}
\begin{aligned}
\int_\omg -\mcL^{L\!P\!S}&u(\xb)z(\xb)\,d\xb\\
= & \frac{C_1 d\left(\lambda - \mu\right)}{(m(\delta))^2}\int_\oomg \left[\int_\oomg \gamma(\left|\mathbf{y}-\mathbf{x}\right|) (\mathbf{y}-\mathbf{x})\cdot \left(u(\mathbf{y}) - u(\mathbf{x}) \right)d\mathbf{y}\right]\times \\
&\hspace{2.4cm}\left[\int_\oomg \gamma(\left|\mathbf{y}-\mathbf{x}\right|) (\mathbf{y}-\mathbf{x})\cdot \left(z(\mathbf{y}) - z(\mathbf{x}) \right)d\mathbf{y}\right] d\mathbf{x}\\[1mm]
+ &  \frac{C_2\mu}{2m(\delta)}\int_\oomg\int_\oomg  \gamma(\left|\mathbf{y}-\mathbf{x}\right|)(u(\mathbf{y})-u(\mathbf{x}))(z(\yb)-z(\xb)) d\mathbf{y}d\mathbf{x}\\
+ & \int_\omgi \mcN^{L\!P\!S}u(\xb)z(\xb)\,d\xb\\
:= &a^{L\!P\!S}(u,z) + \int_{\omg_{loc}}
\mcN^{L\!P\!S}u(\xb)z(\xb)\,d\xb.
\end{aligned}
\end{equation}
It is important to note that the bilinear form $a(\cdot, \cdot)$ induces a norm in the space $V^{L\!P\!S}_\Lambda(\oomg)$, for all $\Lambda\subset\omgi$, or, in other words, $a(u,u)$ is equivalent to $|||u |||^2_{L\!P\!S}$ for all $u\in V^{L\!P\!S}_\Lambda(\oomg)$. Thus, $a(\cdot,\cdot)$ is continuous and coercive.

\smallskip 
By multiplying both equations \eqref{eq:lps-dirichlet} and \eqref{eq:lps-mixed} by a test function and using nonlocal integration by parts, we obtain the following weak formulations.
For $s\in (V^{L\!P\!S})'(\oomg)$, $v_n\in \widetilde V^{L\!P\!S}_{\omg_{nloc}}(\oomg)$, and $w_n\in \widetilde V^{L\!P\!S}_{\omg_{loc}}(\oomg)$,  $u_n\in V^{L\!P\!S}(\oomg)$ is a weak solution of the {\it Dirichlet LPS problem} if $u_n=w_n$ in $\Omega_{loc}$, $u_n=v_n$ in $\Omega_{nloc}$ and
\begin{equation}\label{eq:weak-dir-lps}
a^{L\!P\!S}(u,z) = \int_\omg sz\,d\xb, \qquad \forall \; z\in V^{L\!P\!S}_\omgi(\oomg).
\end{equation}
Similarly, given $s\in (V^{L\!P\!S})'(\oomg)$, $v_n\in \widetilde V^{L\!P\!S}_{\omg_{nloc}}(\oomg)$, and $g_n\in (V^{L\!P\!S})'(\omg_{loc})$,  $u_n\in V^{L\!P\!S}(\oomg)$ is a weak solution of the {\it mixed LPS problem} if $u_n=v_n$ in $\Omega_{nloc}$ and
\begin{equation}\label{eq:weak-mixed-lps}
a^{L\!P\!S}(u,z) = 
\int_{\omg_{loc}} g_n z\,d\xb
+\int_\omg sz\,d\xb, \qquad \forall \; z\in V^{L\!P\!S}_{\omg_{nloc}}(\oomg).
\end{equation}

The well-posedness of \eqref{eq:weak-dir-lps} and \eqref{eq:weak-mixed-lps} follows from the fact that $a^{L\!P\!S}(\cdot,\cdot)$ is continuous and coercive in $V^{L\!P\!S}(\oomg)$ and from the continuity of the right-hand sides. In fact, these properties allow us to apply the Lax-Milgram theorem that guarantees existence and uniqueness of solutions. 

\begin{comment}
$$
\begin{aligned}
W_\delta(\mathbf{u})=&\frac{C_\alpha d\left(\lambda - \mu\right)}{(m(\delta))^2}\int_\Omega \left[\int_{B_\delta (\mathbf{x})} K(\left|\mathbf{y}-\mathbf{x}\right|) (\mathbf{y}-\mathbf{x})\cdot \left(\mathbf{u}(\mathbf{y}) - \mathbf{u}(\mathbf{x}) \right)d\mathbf{y}\right]^2 d\mathbf{x}\\
&+  \frac{C_\beta\mu}{2m(\delta)}\int_\Omega\left[\int_{B_\delta (\mathbf{x})}  \dfrac{K(\left|\mathbf{y}-\mathbf{x}\right|)}{\left|\mathbf{y}-\mathbf{x}\right|^2}\left[\left(\mathbf{u}(\mathbf{y}) - \mathbf{u}(\mathbf{x}) \right)\cdot\left(\mathbf{y}-\mathbf{x}\right)\right]^2 d\mathbf{y}\right]d\mathbf{x},
\end{aligned}
$$
\end{comment}

%%%%%%%%%%%%%%%%%%%%%%%%%%%%%%%%%%%%%%%%%%%
%%%%%%%%%%%%%%%%%%%%%%%%%%%%%%%%%%%%%%%%%%%
\section{Proposed strategies}\label{sec:strategy}
In practice data may only available on the boundary $\partial\oomg$ and not in $\omgi$; in particular, values of the diffusive quantity, for the nonlocal Poisson's equation, and of the displacement, for the LPS model, may be available on parts of $\partial\oomg$, while nonlocal volume constraints may be available on the remaining part of $\omgi$. Thus, as indicated in Figure \ref{fig:domains}, we split the interaction domain in two parts: a ``nonlocal part'', $\omg_{nloc}$, where nonlocal volume constraints are available, and a ``local part'', $\omg_{loc}$, where only local, boundary data are available. As this is not enough for the well-posedness of the problem, we now introduce a strategy that, starting from this incomplete data set, delivers volume constraints on $\omg_{loc}$, hence allowing for the solution of the nonlocal problems. We present our strategies for the nonlocal Poisson equation, as the approach is \underline{identical} for the LPS model (the properties of the method are analyzed for both models).  

\medskip\noindent \textbf{Assumption 1} Only the following data are available:

\smallskip\noindent
{\bf 1.} $w_l\in H^\frac12(\Gamma_{loc})$: {\it local} Dirichlet boundary data on $\Gamma_{loc}=\partial\omg_{loc}\cap\partial\oomg$;\\
{\bf 2.} $v_n\in\widetilde V_{\omg_{nloc}}(\oomg)$: nonlocal Dirichlet data in $\Omega_{nloc}$;\\
{\bf 3.} $s\in V'(\oomg)$: forcing term over $\oomg$.

\smallskip
We design two strategies to automatically convert $w_l$ into a nonlocal volume constraint (either of Dirichlet or Neumann type) on $\omg_{loc}$. As we show in the following section, the most important property of our strategies is their {\it asymptotic compatibility}, i.e.
\begin{equation}\label{eq:AC}
\un \to u_l \;\; {\rm as} \;\; \delta\to 0 \quad {\rm in} \;\; V(\oomg) \;\; {\rm and} \;\; L^2(\oomg).
\end{equation}
Here, $u_n$ is the nonlocal solution corresponding to the proposed nonlocal volume constraints and $u_l$ is the solution of the following Poisson's equation
\begin{equation}\label{eq:local-Dirichlet}
\left\{\begin{array}{ll}
-\Delta \ul = s               & \xb\in\oomg \\[3mm]
\ul = w_l & \xb\in\Gamma_{loc} \\[3mm]
\ul = v_n                     & \xb\in\Gamma_{nloc},
\end{array}\right.
\end{equation}
i.e. the solution of the local problem with boundary data as in {\bf Assumption 1} on $\omg_{loc}$ and with boundary data $v_n|_{\Gamma_{nloc}}$, with $\Gamma_{nloc}=\partial\omg_{nloc}\cap\partial\oomg$. Note that, by prescribing the Dirichlet condition on $\Gamma_{nloc}$ we are assuming that $v_n|_{\Gamma_{nloc}}$ exists and is such that $v_n|_{\Gamma_{nloc}}\in H^\frac12(\Gamma_D)$. We emphasize that we are not we are not assuming $v_n\in H^1(\omg_{nloc})$, but only that $v_n$ has a well-defined trace on $\Gamma_{nloc}$.

%%%%%%%%%%%%%%%%%%%%%%%%%%%%%%%%%%%%%%%%%%%%%%%%%%%
\subsection{Dirichlet-to-Dirichlet strategy} \label{sec:Dirichlet-approach}
The first proposed strategy, referred to as Dirichlet-to-Dirichlet (DtD) strategy, consists in using the local solution $u_l$ of problem \eqref{eq:local-Dirichlet} as Dirichlet volume constraint for the nonlocal problem in $\Omega_{loc}$. We summarize the procedure below.

\medskip\noindent
{\bf 1} Solve the local problem \eqref{eq:local-Dirichlet} to obtain $u_l$. Note that $u_l\in\widetilde V(\omg_{loc})$.

\medskip\noindent
{\bf 2} Solve the (well-posed) nonlocal problem:
\begin{equation}\label{eq:nonlocal-Dirichlet}
\left\{\begin{array}{rlll}
-\ds\mcLn  u_n &=& s & \xb\in\omg \\[3mm]
 u_n          &=& u_l        & \xb\in\Omega_{loc} \\[3mm]
 u_n         &=& v_n        & \xb\in\Omega_{nloc}.
\end{array}\right.
\end{equation}

%%%%%%%%%%%%%%%%%%%%%%%%%%%%%%%%%%%%%%%%%%%%%%%%%%%
\subsection{Dirichlet-to-Neumann strategy} \label{sec:Neumann-approach}
The second strategy, referred to as Dirichlet-to-Neumann (DtN) strategy, consists in using the local solution $u_l$ of problem \eqref{eq:local-Dirichlet} to generate a Neumann volume constraint for the nonlocal problem in $\Omega_{loc}$. We summarize the procedure below.

\medskip\noindent
{\bf 1} Solve the local problem \eqref{eq:local-Dirichlet} to obtain $u_l$. Note that $\mcN^{N\!L} u_l$ for $\xb\in\omg_{loc}$ is well-defined and belongs to $V'(\omg_{loc})$.

\medskip\noindent
{\bf 2} Solve the (well-posed) nonlocal problem:
\begin{equation}\label{eq:nonlocal-Neumann}
\left\{\begin{array}{rlll}
-\ds\mcLn  u_n &=& s & \xb\in\omg \\[3mm]
-\mcN^{N\!L} u_n          &=& -\mcN^{N\!L} u_l        & \xb\in\Omega_{loc} \\[3mm]
 u_n         &=& v_n        & \xb\in\Omega_{nloc}.
\end{array}\right.
\end{equation}

%%%%%%%%%%%%%%%%%%%%%%%%%%%%%%%%%%%%%%%%%%%%%%%%%%%%%%%%
%%%%%%%%%%%%%%%%%%%%%%%%%%%%%%%%%%%%%%%%%%%%%%%%%%%%%%%%
\section{Convergence to the local limit} \label{sec:local-limit}
In this section we study the limiting behavior of the solution as the nonlocal interactions vanish, i.e. as $\delta\to 0$ and we show that \eqref{eq:AC} holds true with a second order convergence rate for both the Poisson's and LPS models. 

\smallskip For both the Dirichlet-to-Dirichlet strategy and Dirichlet-to-Neumann, the following propositions provide bounds for the errors 
\begin{equation}\label{eq:error}
\begin{aligned}
e_{E,N\!L} = |||\un-\ul|||, 
& \quad e_{E,L\!P\!S} = |||\ub_n-\ub_l|||_{L\!P\!S}, \\
e_{0,N\!L} = \|\un-\ul\|_{0,\oomg}, & \quad
e_{0,L\!P\!S} = \|\ub_n-\ub_l\|_{0,\oomg}.
\end{aligned}
\end{equation}

%%%
\smallskip
\begin{theorem}\label{thm:local-limit}
Let $\delta_0\in(0,\infty)$ and $\mcU_l:=\{\ul\in C^4(\oomg): \ul \hbox{ solves \eqref{eq:local-Dirichlet} for }\delta\in(0,\delta_0]\}$ be solutions to \eqref{eq:local-Dirichlet}. Then,
\begin{equation}\label{eq:bound-nl}
e_{E,N\!L} = \mcO(\delta^2).
\end{equation}
\end{theorem}
%%%
\begin{proof}
We only prove \eqref{eq:bound-nl} for the DtD strategy and refer the reader to \cite{DEliaNeumann2019} for the DtN strategy as the steps of the proof are the same. In fact, for DtN, the only difference with the approach presented in that paper is step {\bf 1} (solution of a local problem), where, instead of solving a mixed boundary condition Poisson's problem, we solve a fully Dirichlet problem. 

By definition of $\un$ and $\ul$, we have
\begin{equation}\label{eq:nonlocal-local-Neumann}
\left\{\begin{array}{ll}
-\ds\mcL \un  = s    = -\Delta \ul   & \xb\in\omg \\[3mm]
\un = u_l & \xb\in\omg_{loc} \\[3mm]
\un           = v_n                  & \xb\in\omg_{nloc}.
\end{array}\right.
\end{equation}
We introduce a nonlocal auxiliary problem for the local solution $\ul$, keeping in mind that $v_n$ is compatible with the local solution.
\begin{equation}\label{eq:nonlocal-auxiliary}
\left\{\begin{array}{ll}
-\ds\mcL \ul  = \sl  = -\int_\oomg (\ul(\yb)-\ul(\xb))\gamma(\xb,\yb)\,d\yb & \xb\in\omg \\[3mm]
\ul = u_l & \xb\in\omg_{loc} \\[3mm]
\ul = v_n  & \xb\in\omg_{nloc}.
\end{array}\right.
\end{equation}
In order to estimate $e_{E,N\!L}$ we first consider the point-wise difference 
$s(\xb)\!-\!\sl(\xb)$. Property \eqref{eq:nl-op-difference} implies that
\begin{equation}\label{eq:s-sl}
|s(\xb)-\sl(\xb)| =    \left|\int_\oomg (\ul(\yb)-\ul(\xb))\gamma(\xb,\yb)\,d\yb - \Delta\ul \right|
                  = \mcO(\delta^2).
\end{equation}
Next, we consider the weak forms of \eqref{eq:nonlocal-local-Neumann} and \eqref{eq:nonlocal-auxiliary} and use, in both of them, the test function $z\in V_\omgi(\oomg)$; we have
\begin{equation}\label{eq:weak-wun}
\int_\oomg\int_\oomg (\un(\xb)-\un(\yb))(z(\xb)-z(\yb))\gamma(\xb,\yb)\,d\yb\,d\xb = \int_\omg s\,z\,d\xb,
\end{equation}
\begin{equation}\label{eq:weak-ul}
\int_\oomg\int_\oomg (\ul(\xb)-\ul(\yb))(z(\xb)-z(\yb))\gamma(\xb,\yb)\,d\yb\,d\xb = \int_\omg \sl\,z\,d\xb.
\end{equation}
Subtraction gives
\begin{displaymath}
\begin{aligned}
\int_\oomg\int_\oomg(\un(\xb)-\ul(\xb)-\un(\yb)+\ul(\yb))(z(\xb)-z(\yb))
\gamma(\xb,\yb)\,d\yb\,d\xb = \int_\omg (s-\sl)\,z\,d\xb.
\end{aligned}
\end{displaymath}
To prove the error estimate we then choose $z=\un-\ul\in V_\omgi(\oomg)$. We have
\begin{displaymath}
|||\un-\ul |||^2 \leq \int_\omg (s-\sl)\,(\un-\ul)\,d\xb
\leq \|s-\sl\|_{0,\omg} \|\wun-\ul\|_{0,\omg} \leq \mcO(\delta^2) C_{pn} |||\un-\ul|||.
\end{displaymath}
By dividing both sides by $|||\un-\ul |||$, the error bound follows.
\end{proof}

Before addressing the error bound for the LPS model, we introduce the local problem corresponding to the operator $\mcL_l$ introduced in \eqref{eq:local-lps}, i.e.
\begin{equation}\label{eq:navier}
\left\{\begin{array}{rlll}
-\mcL_l \ub_l &=& \sbb & \quad\xb\in\oomg\\[2mm]
\ub_l &=& \wb_l & \quad\xb \in \Gamma_{loc}\\[2mm]
\ub_l &=& \vb_n & \quad\xb\in\Gamma_{nloc},
\end{array}\right.
\end{equation}
where, $\wb_l$ is the available local Dirichlet data on $\Gamma_{loc}$, and $\vb_n$ is the available nonlocal Dirichlet volume constraint on $\Gamma_{nloc}=\partial\omg_{nloc}\cap\partial\oomg$. As for the local Poisson's equation, we assume that the nonlocal Dirichlet data $\vb_n$ has a well-defined trace on $\Gamma_{nloc}$ and is compatible with the local solution. We can now state the the following theorem, whose proof, based on \eqref{eq:limit-lps}, follows exactly the same steps used in Theorem \ref{thm:local-limit} and is, hence, omitted.

\medskip
\begin{theorem}\label{thm:local-limit-navier}
Let $\delta_0\in(0,\infty)$ and $\mcU^{L\!P\!S}_l:=\{\ub_l\in C^4(\oomg): \ub_l \hbox{ solves \eqref{eq:navier} for }\delta\in(0,\delta_0]\}$ be solutions to \eqref{eq:navier}. Then,
\begin{equation}
e_{E,L\!P\!S} = \mcO(\delta^2).
\end{equation}
\end{theorem}
 
%%%
\begin{remark}\label{rem:L2}
An immediate consequence of Theorem \ref{thm:local-limit} implies that the convergence rate of $e_0$ is at least quadratic. This result can be obtained by applying the Poincar\'e inequality, i.e.
\begin{equation*}
    e_{0,N\!L} = \|\un-\ul\|_{0,\oomg}\leq C_{pn} |||\un-\ul ||| = C_{pn} e_{E,N\!L} = \mcO(\delta^2).
\end{equation*}
Following the same arguments, we can also show that the same bound holds for the LPS model. In fact, paper \cite{mengesha14Navier} provides a Poincar\'e-type inequality associated with the LPS operator $\mcL^{L\!P\!S}$ with constant $C^{L\!P\!S}_{pn}$. Thus, as a consequence of Theorem \ref{thm:local-limit-navier}, we have
\begin{equation*}
    e_{0,L\!P\!S} = \|\ub_n-\ub_l\|_{0,\oomg}\leq C^{L\!P\!S}_{pn} |||\ub_n-\ub_l ||| = C^{L\!P\!S}_{pn} e_{E,L\!P\!S} = \mcO(\delta^2).
\end{equation*}
\end{remark}

%%%%%%%%%%%%%%%%%%%%%%%%%%%%%%%%%%%%%%%%%%%%%%%%%%%%%%%%%%%%%%
\section{Numerical tests}\label{sec:tests}
We report the results of several two-dimensional numerical tests that illustrate our theoretical results and highlight the efficacy of the proposed methods.  

In all tests we utilize {a particle discretizations of the strong form of the nonlocal Poisson's problem and the LPS model introduced in Section \ref{sec:poisson}  and \ref{sec:lps-model} respectively. The meshfree discretization method we use is based on an optimization-based quadrature rule developed and analyzed in \cite{trask2019asymptotically,you2020asymptotically,You_2019,yu2021asymptotically}. In this approach, we discretize the union of the domain and interaction domain, $\oomg$, by a collection of points 
$$\chi_{h} = \{\xb_i\}_{\{i=1,2,\cdots,M\}} \subset \oomg,$$
then solve for the solution $u_{(i)}\approx u_n(\xb_i)$ at $\xb_i\in \chi_h$ using a one point quadrature rule. Although the method can be applied to more general grids, in all numerical tests below we require $\chi_h$ to be a uniform Cartesian grid:
$$\chi_h:=\{(k_{(1)}h,\cdots,k_{(d)}h)|\bm{k}=(k_{(i)},\cdots,k_{(d)})\in\mathbb{Z}^d\}\cap \oomg.$$
Here $h$ is the spatial grid size. To maintain an easily scalable implementation, in our $\delta$-convergence studies \cite{bobaru2009convergence} we assume $h$ to be chosen such that the ratio $\frac{\delta}{h}$ is bounded by a constant as $\delta \rightarrow 0$. This meshfree discretization method based on optimization-based quadrature rules features simplicity in implementation and is asymptotically compatible, i.e., it is such that the nonlocal solution converges to its local counterpart as $\delta,h\rightarrow 0$. For further implementation details, we refer interested reader to \cite{fan2021asymptotically,yu2021asymptotically}.}

%For simplicity here we explain the algorithm for the (mixed) nonlocal Poisson's problem , while a similar idea can be trivially extended to the LPS model following \cite{yu2021asymptotically}.}

%\YY{Discretizing the whole interaction region $\oomg$ by a collection of points 
%$$\chi_{h} = \{\xb_i\}_{\{i=1,2,\cdots,M\}} \subset \oomg,$$
%we aim to solve for the solution $u_{(i)}\approx u_n(\xb_i)$ on all $\xb_i\in \chi_h$. Although the method can be applied to more general grids, in all numerical tests below we require $\chi_h$ to be a uniform Cartesian grid:
%$$\chi_h:=\{(k_{(1)}h,\cdots,k_{(d)}h)|\bm{k}=(k_{(i)},\cdots,k_{(d)})\in\mathbb{Z}^d\}\cap (\oomg).$$
%Here $h$ is the spatial grid size. We pursue a discretization for \eqref{eq:nonlocal-mixed} through the one point quadrature rule:
%\begin{equation}
%\left\{\begin{array}{rl}
%-2\sum_{\xb_j \in \chi_h\cap B_\delta(\xb_i)} \big(u_j-u_i\big) \,\gamma (\xb_i,\xb_j )\omega_{j,i} =  s(\xb_i)   & \xb\in\omg  \\[3mm]
%-\sum_{\xb_j \in \chi_h\cap B_\delta(\xb_i)} \big(u_j-u_i\big) \,\gamma (\xb_i,\xb_j )\omega_{j,i}   \qquad\qquad\qquad\qquad    &\\
%\qquad\qquad=-\sum_{\xb_j \in \chi_h\cap B_\delta(\xb_i)} \big(u_l(\xb_j)-u_l(\xb_i)\big) \,\gamma (\xb_i,\xb_j )\omega_{j,i} & \xb\in\Omega_{loc} \\[3mm]
%u_i         =  v_n(\xb_i) & \xb\in\Omega_{nloc},
%\end{array}\right.
%\end{equation} 
%}

\begin{figure}
\centering
\includegraphics[width=0.7\textwidth]{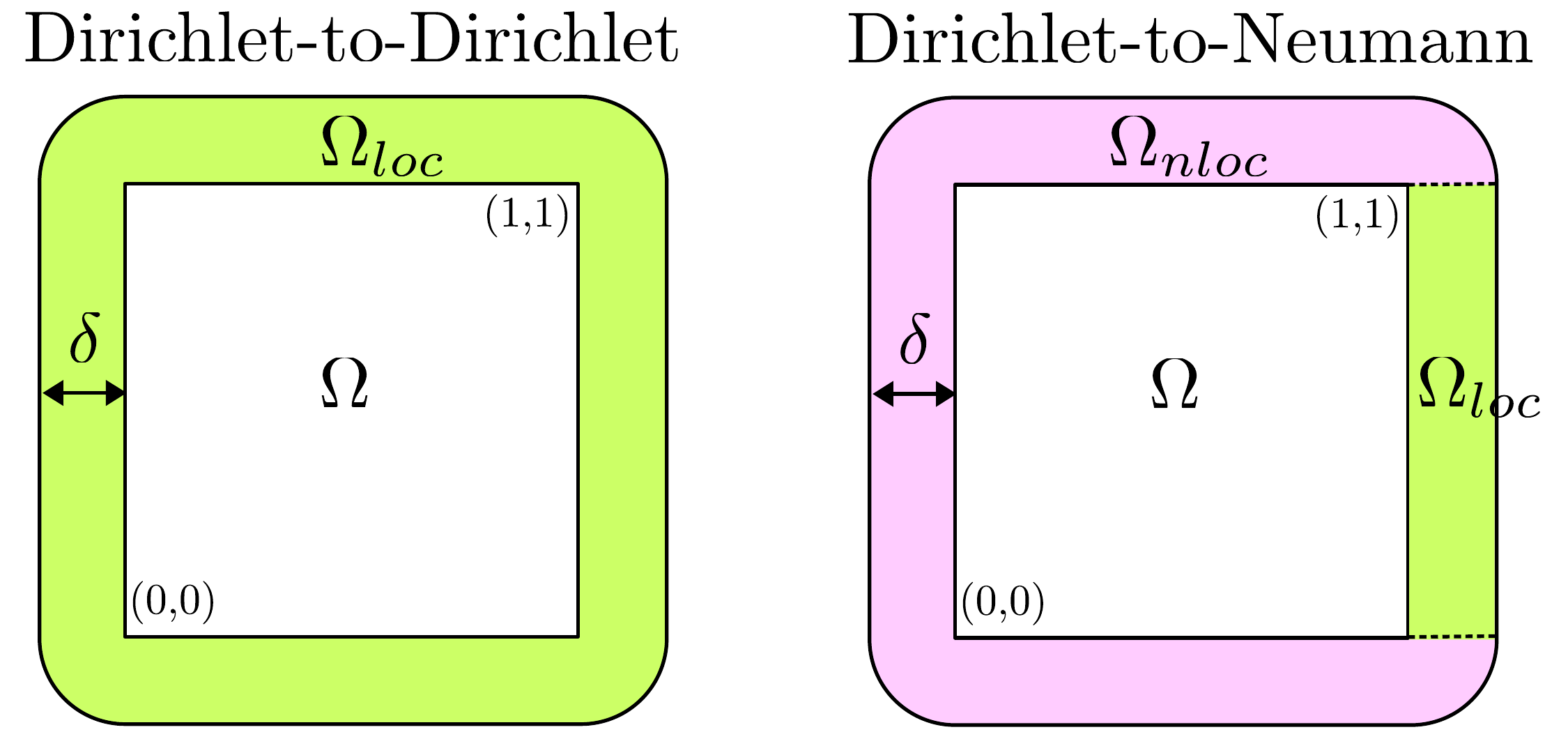}
\caption{Two dimensional configuration utilized in the nonlocal Poisson's consistency and convergence tests for the DtD strategy (left) and DtN strategy (rught).}
\label{fig:ND-tests}
\end{figure}
%%%

%%%%%%%%%%%%%%%%%%%%%%%%%%%%%%%%%%%%%%%%%%
\subsection{Consistency tests for the nonlocal Poisson's equation}
Theorem \ref{thm:local-limit} implies that when the data are smooth enough to have $\mcL^{N\!L}\ul = \Delta u_l$, then $\un = u_l$. We use this observation to conduct a consistency test for the proposed method. Indeed, we consider local solutions $\ul$ such that $\mcL\ul=\Delta\ul$ and expect to observe that the local and nonlocal solutions coincide (up to discretization error).

We refer to the two-dimensional configuration reported in Figure \ref{fig:ND-tests}. Here, $\Omega=(0,1)^2$ and $\Omega_I$ is a layer of thickness $\delta$ surrounding the domain. We use two different configurations for the DtD and DtN strategy. For the former we refer to the configuration on the left of Figure \ref{fig:ND-tests} where $\Omega_I=\Omega_{loc}$; whereas for the latter we refer to the configuration on the right where $\Omega_{loc}$ only covers the right side of the interaction domain, i.e. $\omg_{loc}=[1,1+\delta]\times[0,1]$. In all our consistency tests we use the constant kernel
\begin{equation}\label{eq:const-kernel}
\gamma(\xb,\yb)=\dfrac{4}{\pi \delta^4}\mcX_{B_\delta(\xb)}(\yb)
\end{equation}
and the following set of solutions
\begin{itemize}
    \item $f(\xb)=0$, $u_l(\xb)=\xb_1 + \xb_2$ on $\partial\Omega$, $u_n(\xb)=\xb_1 + \xb_2$ on $\Omega_{nloc}$. Note that this solution corresponds to $u_l=\xb_1+\xb_1$.
    \item $f(\xb)=-6(\xb_1 + \xb_2)$, $u_l(\xb)=\xb_1^3 + \xb_2^3$ on $\partial\Omega$, $u_n(\xb)=\xb_1^3 + \xb_2^3$ on $\Omega_{nloc}$. Note that this solution corresponds to $u_l=\xb_1^3+\xb_2^3$.
\end{itemize}
Consistently with our theory, in both cases and for both strategies (i.e. DtD and DtN) the nonlocal solution coincides with the local solution up to {\it machine precision}. In fact, we observe $e_0\approx\mcO(10^{-17})$. Note that this is possible because our mesh free discretization method can reproduce exactly both linear and cubic polynomials.  

%%%%%%%%%%%%%%%%%%%%%%%%%%%%%%%%%%%%%%%%%%%
\subsection{Convergence tests for the nonlocal Poisson's equation}
We test the convergence of $\un$ to the local solution $u_l$ as $\delta\to 0$. For the same constant kernel defined in \eqref{eq:const-kernel} and for the same configurations illustrated in Figure \ref{fig:ND-tests}, we consider the following set of solutions
\begin{itemize}
    \item $f(\xb)=-2\sin(\xb_1)\cos(\xb_2)$, $\ul(\xb)=\sin(\xb_1)\cos(\xb_2)$ on $\partial\Omega$,\\ $\un(\xb)=\sin(\xb_1)\cos(\xb_2)$ for $\xb\in\Omega_{nloc}$; the corresponding local solution is $\ul(\xb)=\sin(\xb_1)\cos(\xb_2)$.
    \item $f(\xb)=-12(\xb_1^2+\xb_2^2)$, $u_l(\xb)=\xb_1^4+\xb_2^4$ for $\xb\in\partial\Omega$ and $\un(\xb)=\xb_1^4+\xb_2^4$ for $\xb\in\Omega_{nloc}$; the corresponding local solution is given by $\ul=\xb_1^4+\xb_2^4$.
\end{itemize}
Convergence results are reported in Table \ref{tab:ND-DD} for the DtD strategy and in Table \ref{tab:ND-DN} for the DtN strategy. Here, we report, for decreasing values of $\delta$, the $L^2$ norm of the difference between local and nonlocal solution, i.e. $e_0$ and the corresponding rate of convergence. We recall that in our discretization scheme $\delta$ and the node spacing $h$ are related, i.e. their ratio is constant and it is set to 2.5 for the sinusoidal solution and to 3.1 for the polynomial one. In both cases, the smallest $h$ is set to 0.1 and then halved at every run. The observed {\it quadratic} rates are in alignment with our theory, see Remark \ref{rem:L2}. We point out that the faster converge of the DtD strategy is due to the fact that the nonlocal solution is closer (by construction) to the local one. In fact, they coincide on the interaction domain. 
        
\begin{table}[]
    \centering
    \begin{tabular}{|c|c|c|c|c|c|}
    \hline
    \multicolumn{3}{|c|}{sinusoidal} & \multicolumn{3}{c|}{polynomial}\\
    \hline
    $\delta$ & $e_0$ &  rate 
    & $\delta$ & $e_0$ &  rate \\
    \hline
    0.25 &  1.837e-4&  -- & 0.31 & 9.571e-3 & -- \\
    0.125 & 4.443e-5 & 2.0473  & 0.155 & 2.198e-3 & 2.1226  \\
    0.0625 & 1.098e-5  & 2.0174  & 0.0775 & 5.290e-3 & 2.0547 \\
    0.03125 & 2.730e-6 & 2.0071 & 0.0388 & 1.299e-4 & 2.0255 \\
    \hline
    \end{tabular}
    \caption{For the nonlocal Poisson's equation, $L^2$-norm errors and convergence rates for the DtD strategy.}
    \label{tab:ND-DD}
\end{table}

\begin{table}[]
    \centering
    \begin{tabular}{|c|c|c|c|c|c|}
    \hline
    \multicolumn{3}{|c|}{sinusoidal} & \multicolumn{3}{c|}{polynomial}\\
    \hline
    $\delta$ & $e_0$ &  rate 
    & $\delta$ & $e_0$ &  rate \\
    \hline
    0.25 &  2.551e-4 &  -- & 0.31 & 1.094e-2 & -- \\
    0.125 & 7.257e-5  & 1.8136  & 0.155 & 2.929e-3 & 1.9014  \\
    0.0625 & 1.953e-5   & 1.9455  & 0.0775 & 7.720e-4 & 1.9239 \\
    0.03125 & 5.069e-6 & 1.8941 & 0.0388 & 1.990e-4 & 1.9561 \\
    \hline
    \end{tabular}           
    \caption{For the nonlocal Poisson's equation, $L^2$-norm errors and convergence rates for the DtN strategy.}
    \label{tab:ND-DN}
\end{table}

%%%%%%%%%%%%%%%%%%%%%%%%%%%%%%%%%%%%%%%%%%%
\subsection{Numerical tests for the LPS model}
We consider the LPS model introduced in Section \ref{sec:lps-model} and we test consistency and convergence with respect to $\delta$ of both strategies. In all our tests we {consider the deformation of a hollow cylinder as illustrated in Figure \ref{fig:LPS-settings}, and} refer the two-dimensional configurations reported in Figure \ref{fig:LPS-tests} for details on the domain parameters. Specifically, we set $\omg=B_{1.5}({\bf 0})\setminus B_1({\bf 0})$. The interaction domain is then defined as a layer of thickness $2\delta$ surrounding the disc, both inside and outside. For the DtD strategy we use the configuration on the left where $\omg_{loc}=\omgi$, i.e. we assume that only local boundary conditions are available. For the DtN strategy we consider the configuration on the right where $\omg_{loc}$ only corresponds to the inner portion of the interaction domain, i.e. $\omg_{loc}=B_1(\zerob)\setminus B_{1-2\delta}(\zerob)$. 
%%%

\begin{figure}
\centering
\includegraphics[width=0.9\textwidth]{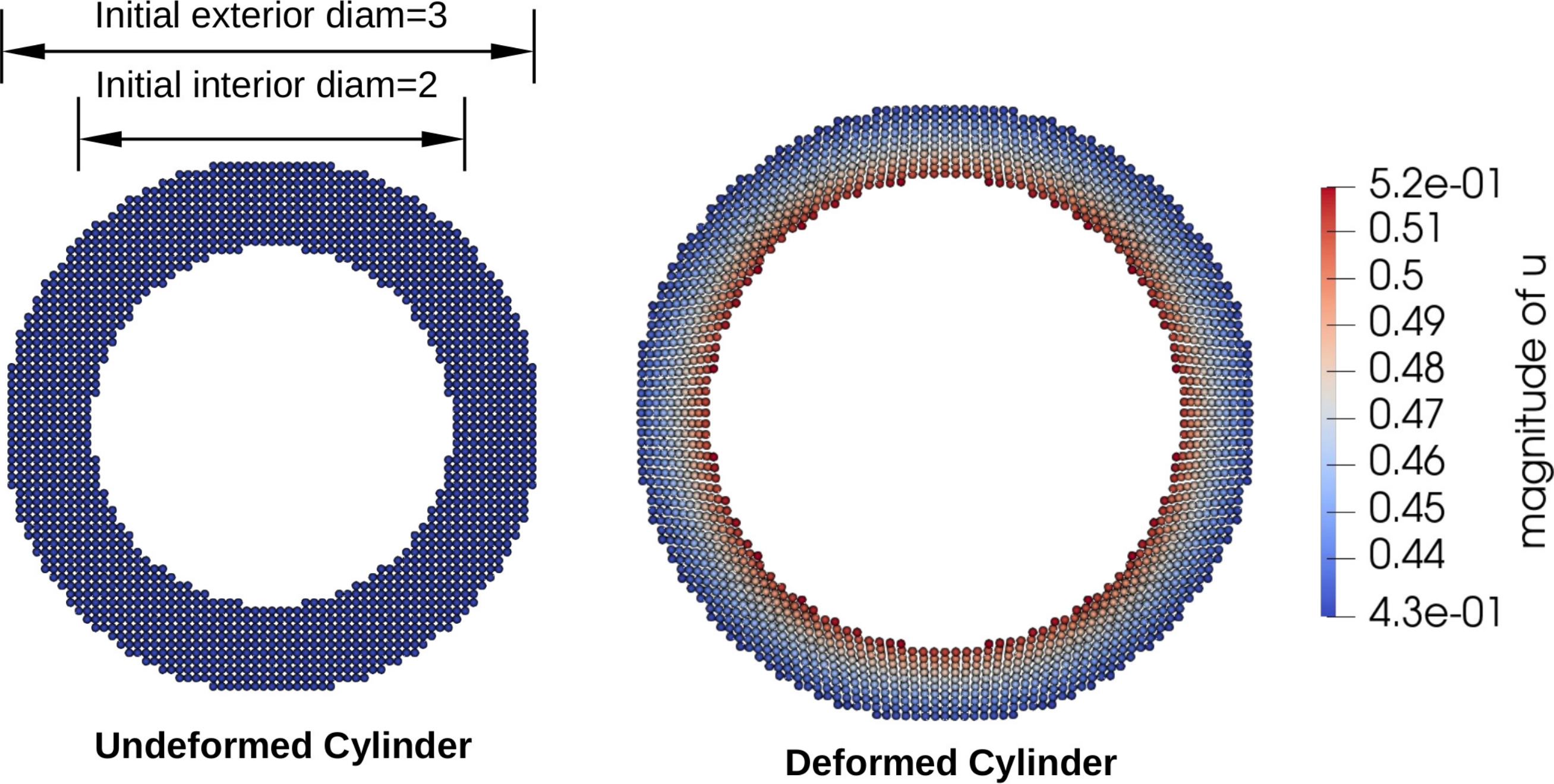}
\caption{Two dimensional hollow cylinder problem settings.}
\label{fig:LPS-settings}
\end{figure}

\begin{figure}
\centering
\includegraphics[width=0.75\textwidth]{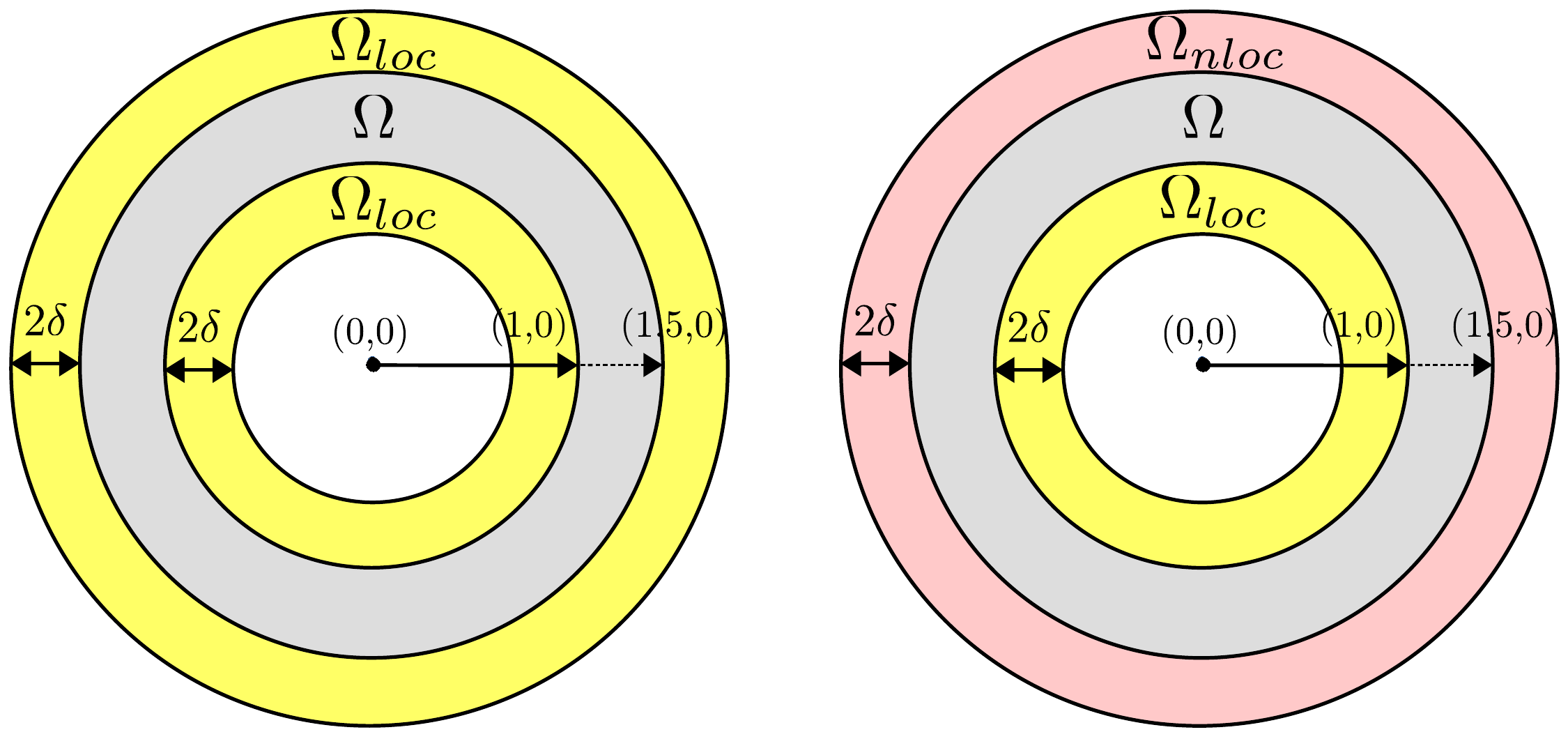}
\caption{Two dimensional configuration utilized in the LPS consistency and convergence tests for the DtD strategy (left) and DtN strategy (right).}
\label{fig:LPS-tests}
\end{figure}

To test the consistency of both procedures, we consider the linear function $\ub_l= {[10\xb_1+2\xb_2,3\xb_1+4\xb_2]}$. This function is such that $\mcL^{L\!P\!S}\ub_l = \mcL_l\ub_l$, where $\mcL^{L\!P\!S}$ and $\mcL_l$ are defined as in \eqref{eq:LPS} and \eqref{eq:local-lps}, respectively. Thus, as for the nonlocal Poisson's model, we expect the nonlocal solution obtained with both the DtD and DtN procedures to be such that $\ub_n=\ub_l$. Our results indicate, once again, that the two solutions are identical, up to {\it machine precision}, i.e. $e_0=\mcO(10^{-17})$. 

\smallskip
To test the convergence with respect to $\delta$ we consider an analytic solution of the local Navier equation \eqref{eq:navier}. {Under a plane strain assumption and subject to an internal pressure $p_0=0.1$, the classical, local displacement solution for the hollow cylinder is given by
$$\ub_l= {\left[A\xb_1+\dfrac{B\xb_1}{\xb_1^2+\xb_2^2},A\xb_2+\dfrac{B\xb_2}{\xb_1^2+\xb_2^2}\right]}$$
where 
$$A=\dfrac{(1+\nu)(1-2\nu)p_0R_0^2}{K(R_1^2-R_0^2)},\;B=\dfrac{(1+\nu)p_0R_0^2R_1^2}{K(R_1^2-R_0^2)}.$$
$R_0=1$ and $R_1=1.5$ are the interior and exterior radius of the (undeformed) hollow cylinder}. We report the results of our tests in Table \ref{tab:lps-DtD} for the DtD strategy and in Table \ref{tab:lps-DtN} for the DtN strategy. In both cases, we consider two values of Poisson's ratio $\nu=0.3$ and 0.49 respectively. Also in this case, the ratio between $\delta$ and $h$ is fixed and set to 3.2; the coarser computational domain is such that $h=0.0937$. The node spacing is then halved at each run of the convergence test. The $L^2$-norm errors show a {\it quadratic} convergence rate, confirming our theoretical predictions in Remark \ref{rem:L2}.

\begin{table}[]
    \centering
    \begin{tabular}{|c|c|c|c|c|c|}
    \hline
    \multicolumn{3}{|c|}{$\nu=0.3$} & \multicolumn{3}{c|}{$\nu=0.49$}\\
    \hline   
    $\delta$ & $e_0$ &  rate 
    & $\delta$ & $e_0$ &  rate \\
    \hline
    0.3 & 4.547e-6 &  -- & 0.3 & 3.253e-5 & -- \\
    0.15 & 7.698e-7  & 2.5625  & 0.15 & 4.836e-6 & 2.7498  \\
    0.075 & 1.714e-7   & 2.1673  & 0.075 & 1.002e-6 & 2.2711 \\
    0.0375 & 4.053e-8 & 2.0801 & 0.0375 & 2.291e-7 & 2.1291 \\
    \hline          
    \end{tabular}           
    \caption{For the LPS model, $L^2$-norm errors and convergence rates for the DtD strategy and different values of Poisson's ratio.}
    \label{tab:lps-DtD}
\end{table}

\begin{table}[]
    \centering
    \begin{tabular}{|c|c|c|c|c|c|}
    \hline
    \multicolumn{3}{|c|}{$\nu=0.3$} & \multicolumn{3}{c|}{$\nu=0.49$}\\
    \hline   
    $\delta$ & $e_0$ &  rate 
    & $\delta$ & $e_0$ &  rate \\
    \hline
    0.3 & 7.651e-6 &  -- & 0.3 & 2.460e-4 & -- \\
    0.15 & 2.025e-6  & 1.9179  & 0.15 & 7.133e-5 & 1.7863  \\
    0.075 & 4.900e-7   & 2.0470  & 0.075 & 1.694e-5 & 2.0737 \\
    0.0375 &  1.111e-7 & 2.1412 & 0.0375 & 3.824e-6 & 2.1478 \\
    \hline                           
    \end{tabular}           
    \caption{For the LPS model, $L^2$-norm errors and convergence rates for the DtN strategy and different values of Poisson's ratio.}
    \label{tab:lps-DtN}
\end{table}

\section{Conclusion}\label{sec:conclusion}
In this work we introduced a technique to automatically convert local boundary conditions into nonlocal volume constraints. A first approximation to the nonlocal solution is provided by the computation of the corresponding local solution, for which local boundary data are available. The local solution is then used to define either a Dirichlet or Neumann nonlocal volume constraints. The latter guarantee that the nonlocal problem is well-posed and that its corresponding solution is physically consistent, i.e. it converges quadratically to the local solution as the nonlocality vanishes. Our conversion method does not have any geometry or dimensionality constraints and is inexpensive compared to the computational cost incurred in when solving nonlocal problems. The theoretical quadratic convergence with respect to the horizon $\delta$ is illustrated by several two-dimensional numerical experiments conducted by meshfree discretization. The consistency, convergence and effectiveness of our approach is demonstrated for both scalar nonlocal Poisson's problems and for nonlocal mechanics problems (namely for the linear peridynamic solid model). 

This work sets the groundwork for the deployment of nonlocal models at the engineering and industry level where the use of such models is often hindered by the technical difficulties that arise when dealing with the lack of volume constraints necessary for the well-posedness and numerical solution of nonlocal equations.

%%%%%%%%%%%%%%%%%%%%%%%%%%%%%%%%%%%%%%%%%%%%%%%%%%%%%%%%%%%%%%%%%%%%%
\section*{Acknowledgments}
M. D'Elia is supported by by the Sandia National Laboratories Laboratory Directed Research and Development program. Sandia National Laboratories is a multi-mission laboratory managed and operated by National Technology and Engineering Solutions of Sandia, LLC., a wholly owned subsidiary of Honeywell International, Inc., for the U.S. Department of Energy’s National Nuclear Security Administration under contract DE-NA0003525. Y. Yu would like to acknowledge support by the National Science Foundation under award DMS 1753031 and Lehigh's High Performance Computing systems for providing computational resources at Sol. X.

This paper, SAND2021-7745 R, describes objective technical results and analysis. Any subjective views or opinions that might be expressed in the paper do not necessarily represent the views of the U.S. Department of Energy or the United States Government. 

\bibliographystyle{siamplain}
\bibliography{references}
\end{document}